\documentclass[11pt,a4paper,oneside,reqno]{amsart}
\pdfoutput=1
\usepackage{amssymb}
\usepackage{esint}
\usepackage{mathtools}

\usepackage[english]{babel}
\usepackage[utf8]{inputenc}
\usepackage[pdftex]{graphicx}
\usepackage{comment}

\usepackage{stix}

\usepackage{tikz-cd}
\usepackage{hyperref}

\theoremstyle{plain}
\newtheorem{theorem}{Theorem}
\newtheorem{lemma}[theorem]{Lemma}

\theoremstyle{definition}
\newtheorem{definition}[theorem]{Definition}

\theoremstyle{remark}
\newtheorem{remark}[theorem]{Remark}

\newcommand{\R}{\mathbb{R}}
\newcommand{\Rplus}{{[0,\infty)}}
\newcommand{\Rminus}{{(-\infty,0]}}

\newcommand{\der}{\mathrm{d}}
\newcommand{\dd}{\,\der}
\newcommand{\eps}{\varepsilon}
\renewcommand{\phi}{\varphi}
\newcommand{\abs}[1]{\left| #1 \right|}

\newcommand{\order}{o}
\newcommand{\Order}{\mathcal{O}}

\newcommand{\rev}[1]{\leftarrowaccent{#1}}
\newcommand{\m}[1]{\mathbf{#1}}
\DeclareMathOperator{\id}{id}
\newcommand{\dummy}{{\,\cdot\,}}

\newcommand{\G}{\mathscr{G}}

\newcommand{\HOX}[1]{}

\title[Reconstruction along a geodesic from sphere data]{Reconstruction along a geodesic from sphere data in Finsler geometry and anisotropic elasticity}
\author{Maarten V. de Hoop}
\thanks{Simons Chair in Computational and Applied Mathematics and Earth Science, Rice University, Houston TX, USA. \texttt{mdehoop@rice.edu}}
\author{Joonas Ilmavirta}
\thanks{Department of Mathematics and Statistics, University of Jyv\"askyl\"a, Finland. \texttt{joonas.ilmavirta@jyu.fi}}
\author{Matti Lassas}
\thanks{Department of Mathematics and Statistics, University of Helsinki, Finland. \texttt{matti.lassas@helsinki.fi}}
\date{\today}


\begin{document}

\begin{abstract}
Dix formulated the inverse problem of recovering an elastic body from the measurements of wave fronts of point sources. We geometrize this problem in the context of seismology, leading to the geometrical inverse problem of recovering a Finsler manifold from certain sphere data in a given open subset of the manifold. We solve this problem locally along any geodesic through the measurement set.
\end{abstract}

\maketitle


\section{Introduction}
\label{sec:intro}

Given small parts of metric spheres in a small open subset of a Finsler manifold, to what extent is the manifold uniquely determined?
We solve it in a local fashion along a reference geodesic.

This is related to the following problem arising from seismology:
There are point sources with known onset times within a manifold and we can detect the wave fronts (corresponding to metric spheres) in a small open set.
Without any assumptions on isotropy, can we use this sphere data to find the material parameters (the stiffness tensor field) within the manifold if densely many point sources are available?

The propagation of the fastest wave fronts or singularities to the elastic wave equation follows the geodesic flow of a Finsler geometry.
This Finsler geometry is not Riemannian unless we make strong symmetry assumptions on the stiffness tensor.
A geometric version of the physical problem was solved in~\cite{Riemann-Dix} in Riemannian geometry, but it is only applicable in an isotropic or elliptically anisotropic situation.
Without assuming any kind of isotropy, we are forced to study a similar geometric problem on Finsler manifolds.

The earth or a subset thereof is modelled as a smooth Finsler manifold $(M,F)$ and the measurement device is an open subset $U\subset M$.
For a point $x\in M$ and a radius $r>0$ a sphere on the manifold is the image of the sphere on~$T_xM$ under the exponential map.
The sphere may intersect itself and fail to be smooth in complicated ways, as we make no assumptions on conjugate points.
These spheres represent the wave fronts of the fastest elastic waves under a mild assumption; see section~\ref{sec:geo-intro}.
The data consists of smooth subsets of these spheres without any knowledge of whether two sphere segments originate from the same source.

The only thing known about the sphere segments is their radius which corresponds to the travel time.
This travel time and the origin time can be assumed to be known if the point sources are artificially produced by sending waves from the set~$U$ and measuring the arrivals of scattered wave fronts from point scatterers.

We prove in theorem~\ref{thm:RU} that this data determines the curvature operator and the Jacobi fields backwards along any geodesic hitting~$U$.
The data consists of spheres, so we may use the associated normal coordinates to describe~$M$ outside~$U$.
In these coordinates the local Riemannian metric (also known as the fundamental tensor) is determined by the sphere data as we show in theorem~\ref{thm:Riemann}.

\subsection{Earlier results}

It was shown in~\cite{Riemann-Dix} that in Riemannian geometry the sphere data (similar to what was considered originally by Dix~\cite{dix}, a prominent seismologist of the 1950s) determines the universal cover of the manifold.
Our main goal and contribution is to make this result more applicable by pushing it into the realm of Finsler geometry.
A crucial component in the proof of both the older Riemannian and the present Finslerian result is setting up a closed ODE system along the geodesic starting normal to the spheres in the data.
We do this in lemma~\ref{lma:JSKR-unique}.

The ODE problems are equivalent in the two geometries; the local Riemannian metric along the reference geodesic makes the Finslerian quantities behave much like their Riemannian counterparts.
The ODE system in~\cite{Riemann-Dix} was set up for the first few derivatives of the inverse shape operator $K(r,t)$.
The version of our lemma~\ref{lma:JSKR-unique} is set up with Jacobi fields and their covariant derivatives instead.
The Jacobi fields are an item of more direct interest than the shape operator and all curvature quantities can be computed directly from them.
The Jacobi fields written in a parallel orthonormal frame encode the components of the local Riemannian metric or the fundamental tensor as we shall see in the proof of theorem~\ref{thm:Riemann}.

Our method of proof is constructive and suggests a computational algorithm that is expected to behave more stably than that extracted from~\cite{Riemann-Dix}.
Unlike this earlier Riemannian version, our ODE system does not suffer from conjugate points at all, and the system can be solved in one go along all of the relevant geodesic.
In this sense we present an improvement and simplification --- in addition to generalization --- of the results of~\cite{Riemann-Dix}.

Inverse problems in elasticity have recently been studied in the framework of Finsler geometry.
See section~\ref{sec:geo-intro} for how Finsler geometry corresponds to elasticity and seismology.
The determination of a Finsler manifold with suitable properties has been shown from boundary distance data~\cite{dHILS:BDF} and broken scattering data~\cite{dHILS:BSR}.

The geometric point of view and Finsler geometry corresponding to the propagation of elastic waves is only beginning to be studied.
Nevertheless, inverse problems for anisotropic elasticity have received substantial attention using different approaches in the past.
Of these we mention problems of identifying inclusions~\cite{INT} or cracks~\cite{II}, tomography for residual stresses~\cite{SW}, and various problems where the stiffness tensor field itself is to be reconstructed~\cite{dHNZ,dHUV,NTU,NU}.
A different geometric approach, using tools of metric rather than smooth geometry, was recently used to approximately reconstruct a manifold from a finite number of seismic sources at unknown times~\cite{dHILS:multisource}.

We wish to highlight two aspects of the importance of studying the unique determination of a Finsler manifold from geophysically relevant data.
First, it is known that in general a Finsler manifold, unlike a Riemannian one, is never uniquely determined by boundary distance data up to isometry~\cite{Iv13}.
It is therefore far from obvious which results on inverse problems on Riemannian manifolds remain true in the Finslerian realm and what additional assumptions on the Finsler manifold could make them remain true.
Second, elastic materials are well described by Finsler geometry and some important features of the theory are lost by restricting to Riemannian descriptions.
A very recent result~\cite{dHILVA:slowness} shows that for a generic stiffness tensor the whole slowness surface and the tensor itself are uniquely determined by any small open subset of the slowness surface for a single polarization.
This kind of data is geometrically accessible; cf. e.g.~~\cite{dHILS:BSR}.
This surprising uniqueness result is false if the slowness surface for any polarization is an ellipsoid, so this strong conclusion relies on the elastic material being described by non-Riemannian Finsler geometry.
For the ultimate goal of indirect measurements of the Earth and other elastic bodies, Finsler geometry provides stronger uniqueness results than Riemannian.

\subsection{The geometric setting}
\label{sec:geom-setup}

This subsection sets up the geometric preliminaries needed to present our results.

\subsubsection{Finsler geometry}
\label{sec:finsler}

A Finsler manifold is a differentiable manifold~$M$ with a function $F_x\colon T_xM\to\Rplus$ for each~$x$.
This function need not be symmetric, but otherwise it satisfies the assumptions of a norm.
Combining the functions on separate fibers gives rise to the Finsler function $F\colon TM\to\Rplus$, which is continuous on~$TM$ and smooth enough on $TM\setminus0$.
Any differentiability assumptions are only assumed to hold on the punctured tangent bundle $TM\setminus0$.

The Finsler function~$F$ could be additionally assumed to be symmetric on each tangent space (reversible), but we do not need to make this assumption.
Finsler functions arising from elasticity are reversible.
It follows from reversibility that the distance function is symmetric and the reverse of a geodesic is a geodesic.
We do not assume reversibility, so these properties do not generally hold.

For any $x\in M$ and $v\in T_xM\setminus0$ we define in local coordinates the matrix
\begin{equation}
\label{eq:g(x,v)-def}
g_{ij}(x,v)
=
\frac12\left.\partial_{v^i}\partial_{v^j} F_x^2\right|_{v}.
\end{equation}
Observe that if~$F_x$ is given by the square root of a positive definite quadratic form (a Riemannian metric on~$T_xM$), then~$g(x,v)$ is independent of~$v$.
In fact,~$g(x,v)$ is independent of~$v$ if and only if the metric is Riemannian.

We call this~$g(x,v)$ or~$g_v(x)$ the \emph{local Riemannian metric}, and it is also known as the fundamental tensor.
Once we have a preferred direction $v\in T_xM\setminus0$, we have a metric tensor~$g(x,v)$ which gives a natural way to linearly identify~$T_xM$ and~$T_x^*M$ and give an inner product on both spaces.
In this paper we work along geodesics, and the preferred direction is the tangent of the geodesic.
This direction dependence is what sets Finsler geometry apart from Riemannian geometry.

The length of any curve is defined using the Riemannian metric associated with its tangent direction.
This gives rise to a distance function $d_F\colon M\times M\to\Rplus$, and length is minimized locally by geodesics like in Riemannian geometry.
If~$F$ is not reversible, then the length $d_F(x,y)$ of the shortest curve from $x\in M$ to $y\in M$ is in general different from $d_F(y,x)$.

We say that the Finsler function $F\colon TM\to\Rplus$ is \emph{fiberwise analytic} if for each $x\in M$ the restriction $F_x\colon T_xM\to\Rplus$ is real-analytic on $T_xM\setminus0$.
We do not need analytic structure on the manifold for this definition; each fiber is a finite dimensional vector space and linear isomorphisms preserve analyticity on such spaces.

Many concepts familiar from Riemannian geometry have their Finsler counterparts.
We use, in particular, geodesics, Jacobi fields, shape operators, surface normal coordinates, and exponential maps.
For more precise definitions and further insights into Finsler geometry, we refer to~\cite{Shen,BCS:finsler-book}.

An important difference between Finsler and Riemannian geometry, besides the directional dependence of the local Riemannian metric mentioned above, is that some symmetry and regularity is lost.
The Finsler distance function is not symmetric, and the exponential map typically fails to be smooth at the origin of a tangent space.

\subsubsection{Operators along a geodesic}
\label{sec:operator-intro}

Let $(M,F)$ be a Finsler manifold and $\gamma\colon I\to M$ a unit speed geodesic defined on an open interval $I\subset\R$.
We assume for convenience that $0\in I$.

The directional curvature operator~$R_{\dot\gamma(t)}$ along~$\gamma$ is a linear operator on~$T_{\gamma(t)}M$.
It is also known as the Riemann curvature in the direction~$\dot\gamma$~\cite{Shen} and we will denote it later by~$R(t)$ for simplicity.
In Riemannian geometry it can be defined by
\begin{equation}
R_{\dot\gamma}(V)=R(V,\dot\gamma)\dot\gamma
,
\end{equation}
where~$R$ is the Riemann curvature tensor.
In Finsler geometry, the matrix elements of~$R_v$ for $v\in T_xM\setminus0$ are given by~\cite[(6.4)]{Shen}
\begin{equation}
\label{eq:S6.4}
(R_v(x,v))^i_k
=
2\frac{\partial \G^i}{\partial x^k}
-v^j\frac{\partial^2\G^i}{\partial x^j\partial v^k}
+2\G^j\frac{\partial \G^i}{\partial v^j\partial v^k}
-\frac{\partial \G^i}{\partial v^j}\frac{\partial \G^i}{\partial v^k}
,
\end{equation}
where (cf.~\cite[(5.2) and (5.7)]{Shen})
\begin{equation}
\begin{split}
\G^i(x,v)
&=
\frac14g^{il}(x,v)\left(
\frac{\partial^2F^2}{\partial x^k\partial v^l}(x,v)v^k
-
\frac{\partial F^2}{\partial x^l}(x,v)
\right)
\\&=
\frac14g^{il}(x,v)\left(
2\frac{\partial g_{jl}}{\partial x^k}(x,v)
-
\frac{\partial g_{jk}}{\partial x^l}(x,v)
\right)v^jv^k
\end{split}
\end{equation}
are the geodesic coefficients.
A Finsler metric is called Berwald if the geodesic coefficients are given by $\G^i(v)=\frac12\Gamma^i_{\phantom{i}jk}(x)v^jv^k$ for some local functions $\Gamma^i_{\phantom{i}jk}(x)$.
Riemannian metrics are Berwald and the local functions in question are the Christoffel symbols.

For any $t\in I$, we denote by
\begin{equation}
N_t
=
\{v\in T_{\gamma(t)}M;g_{\dot\gamma}(\dot\gamma,v)=0\}
\end{equation}
the set of vectors normal to~$\dot\gamma(t)$ in the sense of the inner product~$g_{\dot\gamma(t)}$ on~$T_{\gamma(t)}M$.

Since~$R_{\dot\gamma(t)}$ is self-adjoint with respect to~$g_{\dot\gamma}$ and $R_{\dot\gamma(t)}\dot\gamma(t)=0$ (see e.g. \cite[Section 6.1]{CS:finsler-book}), we have $R_{\dot\gamma(t)}(N_t)\subset N_t$.
Therefore we may consider the curvature operator as a map
\begin{equation}
R(t)
\colon
N_t
\to
N_t
\end{equation}
for all $t\in I$.

A Jacobi field along~$\gamma$ is a vector field~$J$ along~$\gamma$ satisfying the Jacobi equation~\cite[Lemma~6.1.1]{Shen}
\begin{equation}
D_t^2J(t)+R(t)J(t)=0,
\end{equation}
where~$D_t$ is the covariant derivative along~$\gamma$.
As in Riemannian geometry, Jacobi fields correspond to geodesic variations.
The basic properties of Jacobi fields are the same as in Riemannian geometry, including the way they split in parallel and normal components.
We will only study normal Jacobi fields, and they can be characterized as those Jacobi fields~$J$ for which $J(t)\in N_t$ for all $t\in I$.
Then also $D_tJ(t)\in N_t$ for all $t\in I$.

We can thus define the solution operator to the Jacobi equation as the map
\begin{equation}
U(t,s)\colon N_t^2\to N_s^2
\end{equation}
for which $U(t,s)(V,W)=(J(t),D_tJ(t))$ for the Jacobi field~$J$ with initial conditions $J(s)=V$ and $D_tJ(t)|_{t=s}=W$.
Clearly $U(t,t)$ is the identity and $U(t,s)^{-1}=U(s,t)$.

All covariant derivatives and parallel transport are taken with respect to the covariant derivative corresponding to the Chern connection as in~\cite[Section~5]{Shen}.

\subsubsection{Parallel frames}
\label{sec:frame-intro}

Consider again the geodesic $\gamma\colon I\to M$ and recall that $0\in I$.
Take any orthonormal basis $w_1,\dots,w_{n-1},w_n=\dot\gamma(0)$ of~$T_{\gamma(0)}M$ with respect to~$g_{\dot\gamma(0)}$.
Let~$f_k(t)$ be the parallel translation of~$w_k$ along~$\gamma$.
The dual frame has the basis covectors $f^k(t)\in T_{\gamma(t)}^*M$ satisfying $f^i(t)[f_j(t)]=\delta^i_j$ for all $t\in I$.
See e.g. \cite[Chapter 4]{CS:finsler-book} for more details on parallel transport in Finsler geometry.

These frames induce a natural bijection $N_t\to N_0$ given by
$\sum_{k=1}^{n-1}f_k(0)f^k(t)$.
This is an isometry for the inner products~$g_{\dot\gamma}$.
Using this identification, we get the curvature operator
\begin{equation}
\label{eq:hat-R}
\hat R(t)
\colon
N_0\to N_0
\end{equation}
and the Jacobi field evolution operator
\begin{equation}
\label{eq:hat-U}
\hat U(t,s)
\colon
N_0^2\to N_0^2.
\end{equation}
These operators describe the curvature and the evolution of Jacobi field along all of the geodesic in terms of the space~$N_0$ at $t=0$.

\subsection{The inverse problem}

With the geometric prerequisites in place, we can proceed to describe the inverse problem and its solution in detail.

\subsubsection{Sphere data}

Consider a Finsler manifold~$M$ without boundary, and an open set $U\subset M$.
We measure data on the known set~$U$ and aim to determine the unknown remainder of the manifold, $M\setminus U$.
The data consists of all smooth subsets of spheres intersecting~$U$ together with their radii.
We next describe the data in detail.
We assume that we know the geometry fully in~$U$.

Take any point $x\in M$.
A sphere on~$T_xM$ is a level set of of the Finsler function~$F_x$ on this fiber, and the radius is the value of~$F_x$.
A generalized sphere on~$M$ is the image of a sphere on~$T_xM$ under the exponential map $\exp_x\colon T_xM\to M$ based at~$x$.

A forward sphere of radius $r>0$ centered at~$x$ is the set
\begin{equation}
\{y\in M;d_F(x,y)=r\}.
\end{equation}
Since the distance function~$d_F$ may not be symmetric, it is important that the distance is measured from~$x$ to~$y$, not vice versa.
For sufficiently small radii the generalized sphere is in fact a forward sphere because short geodesics minimize distances.

A generalized sphere might not be a smooth hypersurface due to conjugate points.
The differential of the exponential map~$\exp_x$ fails to be bijective precisely at the conjugate locus.
The differential may not exist at zero --- unlike in Riemannian geometry --- but if it does, it is bijective.

Let $C(x)\subset T_xM$ denote the conjugate locus of~$x$.
It consists of all points conjugate to~$x$, not only the first ones along each geodesic.

\begin{definition}
\label{def:vss-TM}
A \emph{visible smooth sphere on~$T_xM$} of radius $r>0$ is an open (in the relative topology of~$T_xM$) connected subset of
\begin{equation}
S^{n-1}_{T_xM}(0,r)\setminus
\left[\exp_x^{-1}(M\setminus U)\cup C(x)\right].
\end{equation}
Here $S^{n-1}_{T_xM}(0,r)=\{v\in T_xM;F_x(v)=r\}$ is the Finsler sphere of radius~$r$.
\end{definition}

Removal of the complement of~$U$ corresponds to ``visibility'' and that of the conjugate locus to ``smoothness''.
This smoothness does not prohibit self-intersections.

\begin{definition}
\label{def:vss-M}
A \emph{visible smooth sphere on~$M$} of radius $r>0$ centered at $x\in M$ is the oriented surface~$\exp_x(S)$, where~$S$ is a visible smooth sphere on~$T_xM$ and so small that~$\exp_x|_{S}$ is injective.
This hypersurface inherits an orientation from the Finsler sphere on~$T_xM$, telling which way is outward (in the direction of the velocity of the geodesic from~$x$ to the point in question).
\end{definition}

We will work on visible smooth spheres locally, so the smallness assumption on~$S$ will be left implicit from now on.
This assumption guarantees that~$\exp_x(S)$ is indeed a smooth surface and we need not worry about cut points.

By~$\nu$ we always denote the outward unit normal.
If $x\in M$ is a point on a visible smooth sphere of radius $r>0$ and~$\nu(x)$ is the normal vector to the sphere, then the center of the sphere is~$\gamma_{x,\nu(x)}(-r)$.
Because the metric is not assumed reversible, it may be that $\gamma_{x,\nu(x)}(-r)\neq\gamma_{x,-\nu(x)}(r)$ since reverse geodesics are not necessarily geodesics.

\begin{definition}
\label{def:SD}
The \emph{sphere data in the set~$U$} is the set
\begin{equation}
\begin{split}
&SD(U,M,F)
=
\{(\Sigma,t);t>0\text{ and}
\\&\qquad
\text{$\Sigma\subset U$ is a visible smooth sphere on $M$ of radius $t$}\}
\end{split}
\end{equation}
of pairs of oriented hypersurfaces and their radii.
Knowledge of the center points of the spheres is not included in the data.
\end{definition}

The centers of the spheres can be considered virtual point sources.
Our data is given as sphere data, and the aim is to reconstruct a Finsler manifold.
The structure is described by a family of coordinate local systems and the Finsler function on them.
It is important that the coordinate systems are based on the sphere data.

\begin{definition}
\label{def:SD-pb}
Let $(M_i,F_i)$, $i=1,2$, be two Finsler manifolds, and $U_i\subset M_i$ open subsets.
If $\psi\colon U_1\to U_2$ is a diffeomorphism, the \emph{pullback of the sphere data} is
\begin{equation}
\psi^*SD(U_2,M_2,F_2)
=
\{(\psi^{-1}(\Sigma),t);(\Sigma,t)\in SD(U_2,M_2,F_2)\}.
\end{equation}
The two manifolds $(M_i,F_i)$ are said to have the same sphere data if there is a diffeomorphism $\psi\colon U_1\to U_2$ so that $SD(U_1,M_1,F_1)=\psi^*SD(U_2,M_2,F_2)$.
\end{definition}

We will assume that the diffeomorphism $\psi\colon U_1\to U_2$ is actually an isometry.
Physically, it makes sense to assume that the measurement area is fully known and only other regions of the manifold are unknown.

\subsubsection{Theorems}

Our main result is the next theorem.
It states that the sphere data determines the curvature operator and Jacobi fields along any geodesic through the measurement set.
The proof only makes use of sphere data for spheres with their centers on the geodesic in question.
The setting
is illustrated in figure~\ref{fig:geodesic-and-normals}.

\begin{theorem}
\label{thm:RU}
Let $(M_i,F_i)$ for $i=1,2$ be two Finsler manifolds without boundary and $U_i\subset M_i$ open subsets.
Suppose there is an isometry $\psi\colon U_1\to U_2$ so that up to identification by~$\psi$ the two manifolds have the same sphere data: $SD(U_1,M_1,F_1)=\psi^*SD(U_2,M_2,F_2)$.

Let $I\subset\R$ be an open interval containing~$0$.
Take any $(x,v)\in TU_1$ with $F_1(x,v)=1$ and a geodesic $\gamma_1\colon I\to M_1$ with the initial data $(x,v)$.
Let~$\gamma_2$ be a geodesic $I\to M_2$ with the initial data $\der\psi(x,v)$.
Identify the spaces $N_0^i=\{v\in T_{\gamma_i(0)}M_i;g_{\dot\gamma_i(0)}(\dot\gamma_i(0),v)=0\}$ for $i=1,2$ by $\der\psi\colon N_0^1\to N_0^2$.

Then for all $t,s\in I\cap\Rminus$ we have
$\hat R_1(t)=\hat R_2(t)$
and
$\hat U_1(t,s)=\hat U_2(t,s)$.
\end{theorem}

\begin{figure}[ht]
    \centering
    \includegraphics[scale=0.35]{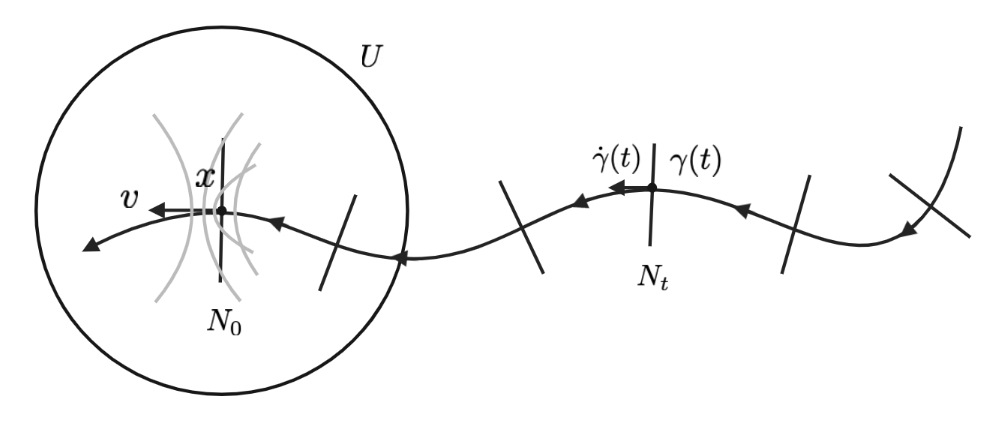}
    \caption{The setting of theorem~\ref{thm:RU}:
    We start from a point $x\in U$ and a unit vector $v\in T_xM$ at time $t=0$ and follow the geodesic backwards into the unknown $M\setminus U$.
    At all times~$t$ we have the normal plane~$N_t$ to~$\dot\gamma(t)$ as a subset of~$T_{\gamma(t)}M$.
    We can identify all these planes canonically with~$N_0$ through parallel transport.
    This identification makes the operators act on just~$N_0$ instead of a bundle of planes along~$\gamma$.
    The data we use is the visible smooth surfaces normal to the geodesic near the initial point~$x$, pictured in light gray.}
    \label{fig:geodesic-and-normals}
\end{figure}

As the map~$\psi$ was assumed to be isometric, the conclusion is trivial as long as the geodesic~$\gamma_1$ remains in~$U_1$.
The theorem states that the curvatures and Jacobi fields agree along all of the geodesic.
The only constraint is that the two geodesics might only be defined a finite amount of time into the past if the geodesics are not maximal or the manifolds are not geodesically complete.
Notice also that as~$\psi$ is an isometry $U_1\to U_2$, its differential $\der\psi\colon TU_1\to TU_2$ maps unit vectors to unit vectors.

Suppose now that the Finsler manifold is geodesically complete.
Given a visible smooth sphere $\Sigma\subset U$, we can define the surface normal exponential map $\phi\colon\Sigma\times\R\to M$ by following geodesics in the normal direction; see section~\ref{sec:snc} and especially~\eqref{eq:sne-def} for more details.
Given any coordinates $\alpha\colon\R^{n-1}\supset\Omega\to\Sigma$, there is an induced map
\begin{equation}
\label{eq:snc-def}
\begin{split}
\phi_\alpha&\colon\Omega\times\R\to M,\\
\phi_\alpha&(z,t)=\phi(\alpha(z),t).
\end{split}
\end{equation}
This map helps give local coordinates on~$M$ in terms of the sphere data.

When~$\phi_\alpha$ is a diffeomorphism from some subset $\Omega'\subset\Omega\times\R$ to the image~$\phi_\alpha(\Omega')$, it gives rise to a natural local vector field~$G$ on~$M$ by pushing forward the constant vector field $(0,1)\in\R^{n-1}\times\R$ on~$\Omega'$.
The integral curves of~$G$ are the geodesics along which~$\phi$ is defined and which come orthogonally to~$\Sigma$.

The next theorem states that the sphere data determines when~$\phi_\alpha$ gives valid local coordinates (which we call the surface normal coordinates).
Furthermore, when we have valid coordinates, the data also determines some properties of the metric in these coordinates.

\begin{theorem}
\label{thm:Riemann}
Let $(M_i,F_i)$ for $i=1,2$ be two Finsler geodesically complete manifolds without boundary and $U_i\subset M_i$ open subsets.
Suppose there is an isometry $\psi\colon U_1\to U_2$ so that up to identification by~$\psi$ the two manifolds have the same sphere data: $SD(U_1,M_1,F_1)=\psi^*SD(U_2,M_2,F_2)$.

Let $\Sigma_1\subset U_1$ be a visible smooth surface and $\alpha_1\colon\R^{n-1}\supset\Omega\to\Sigma_1$ any local coordinates.
Let $\alpha_2=\psi\circ\alpha_1$ be the analogous coordinates on the corresponding visible smooth surface $\Sigma_2=\psi(\Sigma_1)\subset U_2$.
Define the maps $\phi_{\alpha_i}\colon\Omega\times\R\to M_i$ as in~\eqref{eq:snc-def}.

For any $z\in\Omega$ and $t<0$, the differential $\der\phi_{\alpha_1}(z,t)$ is bijective if and only if $\der\phi_{\alpha_2}(z,t)$ is too.
In this case, both $\phi_{\alpha_i}|_{\Omega'}$ define local coordinates on $M_i$ for some neighborhood $\Omega'\subset\Omega\times(-\infty,0)$ of $(z,t)$.
Let~$G_i$ be the vector field corresponding to the surface normal coordinate system as defined above and let $g^i=g^i_{G_i}$ be the Riemannian metric tensor corresponding to it.

Then the local Riemannian metrics agree in these local coordinates: $g^1=g^2$ on~$\Omega'$.
\end{theorem}

\begin{figure}[ht]
    \centering
    \includegraphics[scale=0.5]{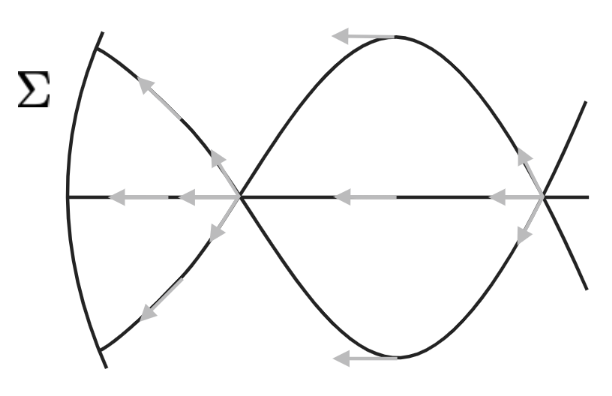}
    \caption{The setting of theorem~\ref{thm:Riemann}:
    The surface normal exponential map of a visible smooth sphere~$\Sigma$ follows geodesics from~$\Sigma$ backwards into the manifold. If the geodesics are lifted to the tangent bundle (indicated with gray arrows), the ``lifted normal exponential map'' $\overline{\phi_\alpha}\colon\Omega\times\R\to TM$ is always an immersion.
    The so defined (local) submanifold of~$TM$ is always smooth but the usual normal exponential map~$\phi_\alpha$ fails to be a local diffeomorphism at the center point of the sphere (on the right) and points conjugate to it (between the center and~$\Sigma$).
    These points are focal to~$\Sigma$.
    At these conjugate points the submanifold of~$TM$ has a small projection to the base but has many points in the same fiber, illustrated by several arrows at the intersection points.
    The gray arrows define the vector field~$G$, and it is a local vector field on~$M$ when the surface normal coordinates actually give local coordinates.}
    \label{fig:sne}
\end{figure}

If the manifolds~$M_1$ and~$M_2$ are Riemannian and they have the same sphere data, then their metric tensors agree in boundary normal coordinates.
In Finsler geometry we can only ever hope to recover the Finsler functions~$F_i$ in directions close to the vector field~$G_i$, as other directions might not correspond to geodesics that come to the measurement sets $U_i\subset M_i$.
It is this directional nature of Finsler geometry that makes it natural to state the results along geodesics rather than on the manifold~$M$.

\begin{remark}
The singular set of~$\phi_{\alpha}$ corresponds to focal points of the visible smooth sphere~$\Sigma$.
Therefore it follows from theorem~\ref{thm:Riemann} that the sphere data determines all the focal distances at all points on all the visible smooth spheres.
See section~\ref{sec:snc} for details and figure~\ref{fig:sne} for an illustration.
\end{remark}

We assumed geodesic completeness for technical convenience.
It can be left out, but then more care is needed in the statement of the theorem to ensure that the relevant geodesics exist on both manifolds.

We point out that the surface normal coordinates of a visible smooth sphere correspond to the normal or polar coordinates of the center point of that sphere.
If we choose the center to be within the known set $U\subset M$, then theorem~\ref{thm:Riemann} gives the local Riemannian metric in these normal coordinates when they are valid.

\begin{remark}
Our main result, theorem~\ref{thm:RU}, is a result along geodesics, and its corollary, theorem~\ref{thm:Riemann}, gives a very local result in small open sets.
By ``very local'' we mean that the local Riemannian metric of the Finsler metric is determined, not the Finsler metric itself in a small open set.
We have chosen to leave out questions of globalization from the present paper.
If the manifold were Riemannian, then theorem~\ref{thm:Riemann} would give the metric tensor in some local coordinate system, thus giving a full description of the local geometry.
In Finsler geometry the information given by the theorem is far from knowing the Minkowski norm on every tangent space in every direction --- it does not even give the norm in a neighborhood of the direction of~$G$ in every tangent space.
The vector field~$G$ is illustrated in figure~\ref{fig:sne}; from the point of view of a single visible smooth sphere and a geodesic through it, we have only access to multiple directions at a point if the point is focal to the visible smooth sphere, but then we do not have a neighborhood in the base.
Therefore the promotion from a single geodesic to a global conclusion is a far bigger step than in Riemannian geometry and it is best taken separately.
Additional assumptions are needed for the conclusion of~\cite{Riemann-Dix}, unique determination of the metric universal cover, to hold.
\end{remark}

Theorem~\ref{thm:RU} is proved in section~\ref{sec:pf-thm-RU} and theorem~\ref{thm:Riemann} in section~\ref{sec:pf-thm-Riemann}.

\subsection{Finsler geometry from seismology}
\label{sec:geo-intro}

A certain class of elastic waves, namely quasi-compressional waves (\textit{qP} waves) follow the geodesics of a Finsler metric.
The Finsler metric arises from the stiffness tensor which describes the anisotropic elastic medium.
The unit cosphere of the Finsler geometry is the \textit{qP} branch of the so-called slowness surface.
In each fiber the Finsler function is given by finitely many elastic parameters and is in fact reversible and fiberwise analytic.
This is very similar to the way every Riemannian metric is fiberwise analytic.
We do not make these assumptions for a few reasons: they turn out to be unimportant, and there may be other physical models where they do not hold.

If one measures in a small open set the wave fronts of \textit{qP}-waves arriving from a point-like event inside the Earth and knows the travel times, then one obtains the sphere data as given in definition~\ref{def:SD}.
Elastic Finsler metrics are in fact reversible, but out of pure geometrical interest we do not make this assumption.

For more details on elastic Finsler geometry, we refer to the discussions in~\cite{dHILS:BDF,dHILS:BSR}.
Finsler geometry in the context of seismology was already mentioned in~\cite{GeoFinsler1,GeoFinsler2,GeoFinsler3,GeoFinsler4}.
For a geophysical discussion of the state of the art of Dix's problem in anisotropic media, see~\cite{Dix99,Riemann-Dix}.
In~\cite{Dix99} points on reflectors are considered as point sources.

\section{Proof of the main theorem}

This section is devoted to the proof of our main result, theorem~\ref{thm:RU}.

\subsection{Curvature, shape, and Jacobi fields}

To set up the tools for proving our results, we expand on the presentation of sections~\ref{sec:operator-intro} and~\ref{sec:frame-intro}.
To keep notation as simple as possible, we denote covariant derivatives along the geodesic~$\gamma$ by~$D_t$ instead of~$\nabla_{\dot\gamma(t)}$.

In terms of the frame (see section~\ref{sec:frame-intro}), the directional curvature operator~$R(t)$ along the Finsler geodesic is given by
\begin{equation}
R(t)
=
\sum_{i,k=1}^{n-1}
(R_{\dot\gamma(t)}(\gamma(t),\dot\gamma(t)))^i_kf_i(t)f^k(t).
\end{equation}
Using the fixed frame, we can also regard~$R(t)$ simply as a matrix depending on~$t$.
This corresponds to the operator~$\hat R(t)$ of~\eqref{eq:hat-R} upon identifying~$N_0$ with~$\R^{n-1}$ through an orthonormal basis.
But we shall drop these identifications and the hat and simply consider~$R(t)$ as a time-dependent matrix.

Our reconstruction algorithm works in the reverse direction along a geodesic because we want the original measured signal to come in the forward direction.
The signs will be different from the related Riemannian result~\cite{Riemann-Dix}.

Consider the forward sphere of radius $r-t>0$ centered at~$\gamma(t)$. 
If~$\gamma(t)$ and~$\gamma(r)$ are not conjugate along~$\gamma$, then this defines locally a smooth surface~$\Sigma(r,t)$ near~$\gamma(r)$ and varying~$r$ with~$t$ fixed foliates a neighborhood of~$\gamma(r)$ with the hypersurfaces~$\Sigma(r,t)$.
Let~$\nu$ be the normal vector field of these surfaces oriented in the same direction as~$\dot\gamma(r)$.
It is a well defined smooth vector field near~$\gamma(r)$ and has unit length at each point.

We define a Riemannian metric~$\hat g$ in this small neighborhood of~$\gamma(r)$ by letting $\hat g(x)=g_{\nu(x)}(x)$.
With respect to this Riemannian metric we can define the shape operator~$S(r,t)$ of~$\Sigma(r,t)$ at~$\gamma(r)$ as we would in the Riemannian case.
More explicitly, the action of $S(r,t)$ on $w\in T_{\gamma(r)}\Sigma(r,t)=\dot\gamma(r)^\perp\subset T_{\gamma(r)}M$ is $S(r,t)w=\nabla_w\nu$, the covariant derivative of the normal field with respect to the Riemannian metric corresponding to the normal vector field.
Expressed in terms of the frame, $S(r,t)$ is a matrix depending on two parameters.

Notice that the very construction of~$S(r,t)$ depends on the two points not being conjugate.
Difficulties brought by conjugate points are an important aspect of our inverse problem.
However, we have diminished the role of conjugate points in our approach in comparison to the Riemannian one in~\cite{Riemann-Dix}.
In lemma~\ref{lma:JSKR-unique} we will solve Jacobi fields from a system of ODEs and the Jacobi fields are well-defined everywhere, whereas in~\cite{Riemann-Dix} the inverse shape operator~$K$ was solved, and it is not defined everywhere if there are conjugate or focal points.

The operators~$R(t)$ and~$S(r,t)$ are related via the Riccati equation~\cite[Lemma~14.4.2]{Shen}
\begin{equation}
\label{eq:riccati}
\nabla_{\dot\gamma(r)}
S(r,t)+S(r,t)^2+R(r)
=
0,
\end{equation}
which holds for any $r,t\in\R$ for which the corresponding points on~$\gamma$ are not conjugate.
Here the covariant derivative~$\nabla$ of the local Riemannian metric associated with the family of geodesics that give rise to the surface~$\Sigma(r,t)$.
For simplicity, we will write
$
\nabla_{\dot\gamma(r)}
S(r,t)
=
D_r
S(r,t)
$.

Let us then study the asymptotics of~$S(r,t)$ near $r=t$.
We denote $K(r,t)=S(r,t)^{-1}$ whenever~$S(r,t)$ is invertible.
Invertibility can indeed fail:
If one equips the sphere~$S^n$ with the usual round Riemannian metric, then $S(r,t)=0$ when $r=t+\frac\pi2$, and this occurs before the first conjugate point ($r=t+\pi$).

\begin{lemma}
\label{lma:K-taylor}
For~$r$ sufficiently close to~$t$ but $r\neq t$, the shape operator~$S(r,t)$ is invertible and we have
\begin{equation}
\label{eq:K-taylor}
K(r,t)
=
(r-t)\id
+\frac{(r-t)^3}3R(r)
+\order(\abs{r-t}^3).
\end{equation}
\end{lemma}

\begin{proof}
By~\cite[Theorem~14.4.3]{Shen} we have
\begin{equation}
\label{eq:S-taylor}
S(r,t)
=
(r-t)^{-1}\id
-\frac{(r-t)}3R(r)
+\order(\abs{r-t}).
\end{equation}
We may replace~$R(r)$ with~$R(t)$, as this only induces an error $\Order(\abs{t-r})$.
The invertibility result follows immediately.
Using Neumann series for the inverse gives~\eqref{eq:K-taylor}.
\end{proof}

\subsection{Determination of curvature operators and Jacobi fields}

For any $t\in\R$, let~$J(\dummy,t)$ be a linearly independent family of $n-1$ normal Jacobi fields along~$\gamma$ which vanish at~$\gamma(t)$.
We can think of~$J(r,t)$ as an $(n-1)\times(n-1)$ matrix whose every column is a Jacobi field with respect to the variable~$r$.
Geometrically, these Jacobi fields correspond to geodesics emanating from~$\gamma(t)$ in directions close to~$\dot\gamma(t)$.

Every Jacobi field satisfies the Jacobi equation, and that can be written collectively as
\begin{equation}
\label{eq:Jacobi}
D_r^2
J(r,t)+R(r)J(r,t)
=
0.
\end{equation}
The Jacobi fields also satisfy a first-order equation as the next lemma shows.

These requirements do not determine the matrix-valued function~$J(r,t)$ uniquely.
It is only unique up to multiplication from the right by an invertible matrix depending on~$t$.
The matrix nature of our operators will be explained in more detail soon.

\begin{lemma}
\label{lma:1order-Jacobi}
The Jacobi fields satisfy
\begin{equation}
\label{eq:DJ=SJ}
D_r
J(r,t)
=
S(r,t)J(r,t)
\end{equation}
whenever the shape operator~$S(r,t)$ is defined.
\end{lemma}

\begin{proof}
The shape operator was only defined when~$\gamma(t)$ and~$\gamma(r)$ are not conjugate along~$\gamma$.
The curvature operator~$R(t)$ and the shape operator~$S(r,t)$ are related via the Riccati equation~\eqref{eq:riccati}, 
which holds for any $r,t\in\R$ for which the corresponding points on~$\gamma$ are not conjugate.

It follows from standard ODE theory that for any fixed~$t$ there is a solution to~\eqref{eq:DJ=SJ} and the solution is unique up to change of basis if~$J$ is assumed invertible at some point.
The shape operator annihilates vectors parallel to the reference geodesic, whence the rows of a solution~$J$ stay normal to~$\dot\gamma(r)$ if they are normal at some point.

If a matrix~$\tilde J(r,t)$ whose columns are linearly independent and normal to~$\dot\gamma(r)$ satisfies~\eqref{eq:DJ=SJ}, then it is easy to check using~\eqref{eq:riccati} that it also satisfies~\eqref{eq:Jacobi}.
Since Jacobi fields stay bounded as $r\to t$ and the asymptotic formula~\eqref{eq:S-taylor} indicates that~$S(r,t)$ blows up in this limit, we must have $\tilde J(r,t)=\Order(\abs{r-t})$ and thus $\tilde J(r,r)=0$.
For any $t\in\R$ there is an $(n-1)$-dimensional space of normal Jacobi fields vanishing at~$\gamma(t)$ and~$J(\dummy,t)$ was defined so that it spans this space.
Therefore, up to the freedom of changing basis on this space, the functions~$\tilde J(\dummy,t)$ and~$J(\dummy,t)$ agree.
In particular, the Jacobi fields~$J(r,t)$ solve~\eqref{eq:DJ=SJ}.
\end{proof}


Let us now make our use of the frames explicit for the sake of concreteness.
There are unique matrix-valued functions~$\m j(r,t)$, $\m s(r,t)$ and~$\m r(r)$ so that
\begin{equation}
J(r,t)
=
\sum_{j,k=1}^{n-1}
\m j_{jk}(r,t)f_j(r)e^k,
\end{equation}
where $e^k\in\R^{n-1}$ are the Euclidean basis vectors, and
\begin{equation}
S(r,t)
=
\sum_{j,k=1}^{n-1}
\m s_{jk}(r,t)f_j(r)f^k(r)
\end{equation}
and
\begin{equation}
R(r)
=
\sum_{j,k=1}^{n-1}
\m r_{jk}(r)f_j(r)f^k(r).
\end{equation}
Similarly, when the shape operator is invertible, we let $\m k(r,t)=\m s(r,t)^{-1}$, corresponding to the operator~$K(r,t)$ defined above.
The matrix-valued shape operator~$\m s(r,t)$ is not defined when the corresponding two points on the geodesic are conjugate, just like with~$S(r,t)$.

One could freely identify $R(t)=\m r(t)=\hat R(t)$ and similarly for the other quantities, but we find that the distinction adds some clarity to the proof.

Let us reiterate the geometrical interpretation of~$J(r,t)$:
For any $t\in\R$ and $v\in\R^{n-1}$, the vector field $r\mapsto \sum_{j,k=1}^{n-1}\m j_{jk}f_j(r)v_k$ is a normal Jacobi field along~$\gamma$ which vanishes at $r=t$.
The Jacobi fields corresponding to linearly independent choices of~$v$ are linearly independent.
Since the Jacobi fields are linearly independent and vanish at $r=t$, it follows that the matrix~$\partial_r\m j(r,t)$ is invertible at and near $r=t$.

Let us rewrite our key identities in matrix form.
The second-order equation for the Jacobi fields (cf.~\eqref{eq:Jacobi}) is
\begin{equation}
\label{eq:Jacobi-m}
\partial_r^2\m j(r,t)+\m r(r)\m j(r,t)
=
0.
\end{equation}
The first-order equation for the Jacobi fields (cf.~\eqref{eq:DJ=SJ}) is
\begin{equation}
\label{eq:DJ=SJ-m}
\partial_r\m j(r,t)
=
\m s(r,t)\m j(r,t).
\end{equation}
The asymptotic expansions for~$\m k$ near $r=t$ (cf.~\eqref{eq:K-taylor}) is
\begin{equation}
\label{eq:K-taylor-m}
\m k(r,t)
=
(r-t)I
+\frac{(r-t)^3}3\m r(r)
+\order(\abs{r-t}^3).
\end{equation}
Now we are ready to prove a uniqueness result for Jacobi fields and curvature operators.

\begin{lemma}
\label{lma:JSKR-unique}
Let $I_0\subset I_1\subset\R$ be two nested open intervals.
Knowing the shape operator $S(r,t)$ for all $r\in I_0$ and $t\in I_1$ for which it is defined determines uniquely
\begin{enumerate}
\item the Jacobi fields~$J(r,t)$ up to multiplication by a $t$-dependent invertible matrix from the right,
\item the shape operator~$S(r,t)$ when it is defined,
\item the inverse shape operator~$K(r,t)$ when it is defined, and
\item the curvature operator~$R(t)$
\end{enumerate}
for all $r,t\in I_1$.
\end{lemma}

\begin{proof}
We will prove the result in matrix form as it is more convenient.
We assume that $I_0=(-\eps,\eps)$ and $I_1=(-T,\eps)$ for $0<\eps<T$, so that we prove ``uniqueness to the reverse direction''.
This direction is relevant for our application, and the proof for the other direction is identical.
Combining the two directions gives the full statement.
By translation, we may assume~$I_0$ and~$I_1$ to be of the given form, and at the end one may let $T\to\infty$ to cover the case when~$I_1$ is unbounded.

By making $\eps>0$ sufficiently small we can ensure that the shape operator~$S(r,t)$ is well defined when $(r,t)\in(-\eps,\eps)^2$ and $r\neq t$.
This amounts to saying that the initial data is set on such a short interval that no focal points to the relevant spheres appear.

{\bfseries Step 1: Setting up a system of ODEs.}
For $k\in\{0,1,2,3\}$, let us denote
\begin{equation}
\m j_k(r,t)=\partial_t^k\m j(r,t)
\end{equation}
and
\begin{equation}
\m y_k(r,t)=\partial_r\partial_t^k\m j(r,t).
\end{equation}
Using~\eqref{eq:Jacobi-m} we get the equations
\begin{equation}
\label{eq:eom1}
\begin{cases}
\partial_r\m j_k(r,t)
=
\m y_k(r,t)
\\
\partial_r\m y_k(r,t)
=
-\m r(r)\m j_k(r,t)
\end{cases}
\end{equation}
for all four values of~$k$.

The crucial step is to compute the curvature matrix~$\m r$ in terms of~$\m j_k$ and~$\m y_k$.
Recall that the matrix~$\m y_0(r,t)$ is invertible in a neighborhood of the diagonal $r=t$.
Combining~\eqref{eq:DJ=SJ-m} and the definition of~$\m k(r,t)$ as the inverse of~$\m s(r,t)$ when invertible, we find
\begin{equation}
\label{eq:k=j/y}
\m k(r,t)
=
\m j_0(r,t)
\m y_0(r,t)^{-1}.
\end{equation}
On the other hand, the asymptotic formula~\eqref{eq:K-taylor-m} indicates that
\begin{equation}
\label{eq:D3K=R-m}
\partial_t^3\m k(r,t)|_{t=r}
=
-2\m r(r).
\end{equation}
Combining~\eqref{eq:k=j/y} and~\eqref{eq:D3K=R-m}, we find
\begin{equation}
\label{eq:R=Q-m}
\m r(r)
=
-\frac12 Q(\m j_1(r,r),\m j_2(r,r),\m j_3(r,r),\m y_0(r,r),\m y_1(r,r),\m y_2(r,r)),
\end{equation}
where the function~$Q$ is defined via
\begin{equation}
\label{eq:Q-def}
\begin{split}
&
Q(A_1,A_2,A_3,B_0,B_1,B_2)
\\&\quad=
A_3B_0^{-1}
-3A_2B_0^{-1}B_1B_0^{-1}
\\&\qquad
+6A_1B_0^{-1}B_1B_0^{-1}B_1B_0^{-1}
-3A_1B_0^{-1}B_2B_0^{-1}.
\end{split}
\end{equation}
We used the fact that~$\m j_0(r,r)=0$ to simplify this expression.
This function is only well defined when the matrix~$B_0$ is invertible, and~$\m y_0(r,t)$ is indeed invertible in a neighborhood of the diagonal.
Observe that this~$Q$ is invariant under multiplication of all the input matrices from the right by the same invertible matrix.

The ODE system~\eqref{eq:eom1} can now be recast as
\begin{equation}
\label{eq:eom2}
\begin{cases}
\partial_r\m j_k(r,t)
=
\m y_k(r,t)
\\
\partial_r\m y_k(r,t)
=
\frac12 Q(\m j(r,r),\m y(r,r))\m j_k(r,t),
\end{cases}
\end{equation}
where $k=0,1,2,3$ and we have abbreviated the notation for~$Q$ for legibility.
This is a closed system of ODEs for the matrix-valued functions~$\m j_k$ and~$\m y_k$.
We will prove unique solvability of this system.

{\bfseries Step 2: Setting up initial conditions.}
The process of setting up the initial conditions is summarized in figure~\ref{fig:initial-conditions}.

\begin{figure}
    \centering
    \includegraphics[scale=.75]{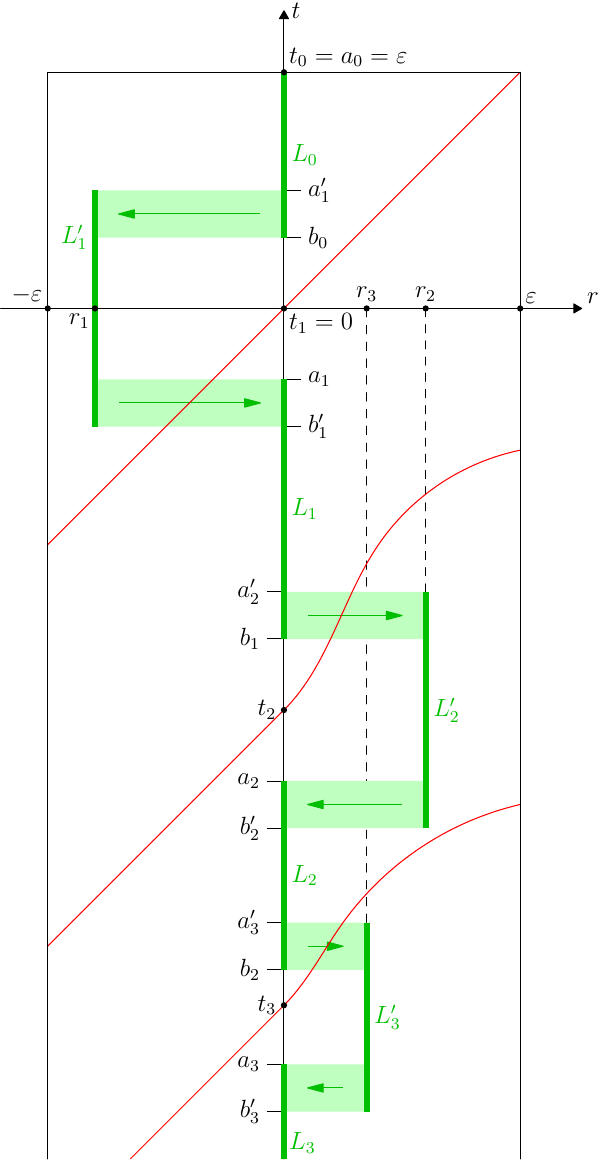}
    \caption{Setting up the initial conditions begins by setting the matrix~$\m j$ to be the identity on the \textcolor{green}{line~$L_0$}. The derivative~$\m y$ is determined by the shape operator included in the data. We then avoid the \textcolor{red}{conjugate locus} as follows. Instead of using the whole $t$-axis which contains conjugate points at $t=t_i$, we only use the \textcolor{green}{segments~$L_i$} on that axis. To work around the conjugate point at $(0,0)$, we take the values of~$\m j$ and~$\m y$ on the lower part of~$L_0$ and solve the Jacobi equation to propagate the data to the upper part of the~$L_1'$, which is depicted as light green with an arrow. We extend the values of~$\m j$ smoothly to the whole of the \textcolor{green}{shifted line~$L_1'$} and choose~$\m y$ as required by the data. Once we are past the conjugate point, we repeat the same process to get back to the $t$-axis. This is repeated at all conjugate points. Solving the Jacobi equation behaves well across conjugate points, whereas smooth extension does not. This way we got to prescribe the initial data smoothly without worrying about consistent singular behaviour at conjugate points.}
    \label{fig:initial-conditions}
\end{figure}

Let us now see why the data determines~$\m r|_{I_0}$.
For $r,t\in I_0$, the shape operator matrix~$\m s(r,t)$ is known when it exists, and it always exists and is invertible when~$r$ and~$t$ are close enough but different; see lemma~\ref{lma:K-taylor}.
The inverse is denoted by~$\m k(r,t)$.
When extended continuously to the diagonal $r=t$ where~$\m s$ is not defined, it is not invertible, but it is continuous and several times differentiable across the diagonal.
The data therefore determines~$\m k(r,t)$ for $r,t\in I_0$ when $\abs{r-t}$ is small enough.
By~\eqref{eq:D3K=R-m} this determines~$\m r|_{I_0}$.

Let us then study setting initial conditions at $r=0$.
We denote
\begin{equation}
C_r
=
\{t\in I_1;\text{$\gamma(t)$ and~$\gamma(r)$ are conjugate or $t=r$}\}.
\end{equation}
For any~$r\in I_1$ the set $C_r\subset I_1$ is discrete since conjugate points cannot accumulate.
The smooth sphere~$\Sigma(r,t)$ exists if and only if $t\notin C_r$; thus the set~$C_r$ for $r\in I_0$ is determined by the data --- specifically the holes in the data.

We are free to choose initial conditions for the Jacobi fields~$\m j(r,t)$ under some constraints:
\begin{itemize}
\item Because $\m j(r,r)=0$, we need to have $\m j(r,t)=\Order(\abs{r-t})$.
\item The matrix $\m j(r,t)$ is invertible if and only if $t\notin C_r$.
\item We want the initial conditions to be regular enough so that $\m j\in C^4$ is possible.
\end{itemize}
To achieve this, we will first set our initial values in a set that avoids conjugate points.

Let us denote $U=I_0\times I_1$ and $C=\bigcup_{r\in I_0}\{r\}\times C_r\subset U$.
The set $C\subset U$ is closed, which follows from continuity of the Jacobi field matrix~$\m j(r,t)$.
In the statement of the theorem, we assumed that $\m s|_{U\setminus C}$ is known.

We will now construct a set of line segments in $U\setminus C$ on which we can set initial values easily.
There are numbers $t_1>t_2>\dots\in\bar I_1$ so that $C_0=\{t_1,t_2,\dots,t_N\}$.
Here the number~$N$ of points may be one, infinity, or any number in between.
Since $0\in C_0$, we have $N\geq1$.
By making $\eps>0$ smaller if necessary, we may assume that $-\eps\notin C_0$.
Let us denote $t_0=-\eps$.
For $\eps>0$ small enough we have in fact $t_1=0$, but we need not assume this.

If $N<\infty$, we let $t_{N+1}=-T$.
To simplify the calculation, we may assume without loss of generality that $-T\notin C_0$.
Namely, we can replace~$T$ with $T-\delta$ for some small $\delta>0$ if necessary to get $-T\notin C_0$ and then get the full result by letting $\delta\to0$ in the end.

For any $i\geq1$, there is $r_i\in I_0$ so that $(r_i,t_i)\notin C$ and the points~$\gamma(0)$ and~$\gamma(r_i)$ are not conjugate.
The second non-conjugacy condition can also be achieved by taking $\eps>0$ so small that there are no conjugate points on~$\gamma|_{I_0}$.
There are numbers $a_i',b_i'\in\R$ so that
\begin{equation}
\frac12(t_{i-1}+t_i)
>
a_i'
>
t_i
>
b_{i}'
>
\frac12(t_i+t_{i+1})
\end{equation}
and $\{r_i\}\times[b_i',a_i']\cap C=\emptyset$.
Then we pick any $b_{i-1}\in(t_i,a_i')$ and $a_i\in(b_i',t_i)$.

We set $a_0=t_0=-\eps$.
If $N<\infty$, we let $b_N=-T$.

With this construction we have found disjoint line segments $L_i=\{0\}\times[b_i,a_i]$, $0\leq i<N+1$, and $L_i'=\{r_i\}\times[b_i',a_i']$, $0<i<N+1$.
Each of these segments has positive distance to the conjugate set $C\subset U$, and the union of projections to the $t$-axis is~$\bar I_1$.
Any $t\in I_1$ is in the projection of one or two such segments.
The shape operator is well defined in a neighborhood of each of these line segments.

Now we can start setting the initial values for our Jacobi fields.
First, we simply let $\m j|_{L_0}=I$.
Since~$\m s$ is defined in a neighborhood of~$L_0$, the condition~\eqref{eq:DJ=SJ-m} determines~$\partial_r\m j|_{L_0}$.
With the initial value and derivative given, the Jacobi equation~\eqref{eq:Jacobi-m} determines~$\m j$ on $I_0\times[b_0,a_0]$, as~$\m r|_{I_0}$ is determined from the data as explained above.

In particular, the Jacobi field matrix~$\m j$ is determined on part of~$L_1'$, namely $\{r_1\}\times[b_0,a_1']$.
This part is uniquely determined, invertible and smooth because the two points~$\gamma(0)$ and~$\gamma(r_1)$ are not conjugate.
Now we extend~$\m j$ smoothly to the whole line segment~$L_1'$, keeping it invertible at every point.
The derivative~$\partial_r\m j|_{L_1'}$ is then determined by~\eqref{eq:DJ=SJ-m}.
Using the Jacobi equation~\eqref{eq:Jacobi-m} again, this information on~$L_1'$ determines~$\m j$ on an initial part of~$L_1$, and we continue it smoothly and invertibly to the whole~$L_1$.
Continuing iteratively and using the knowledge of the shape operator in $U\setminus C$, we find the values of~$\m j(0,t)$ and~$\partial_r\m j(0,t)$ for all $t\in I_1$.

We have fixed our consistent initial data corresponding to a well behaved Jacobi field matrix~$\m j(r,t)$, so we no longer need to avoid conjugate points.

{\bfseries Step 3: Uniqueness of the solution.}
We then set out to prove uniqueness of solutions to the ODE system~\eqref{eq:eom2} with our initial data for~$\m j_k(0,t)$ and~$\m y_k(0,t)$ for all $k\in\{0,1,2,3\}$ and $t\in I_1$.
We already know existence, so it suffices to show uniqueness.
One could show existence with similar arguments.

To abbreviate our notation, let us denote
\begin{equation}
\begin{split}
Z(r,t)
&=
(\m j_0(r,t),\m j_1(r,t),\m j_2(r,t),\m j_3(r,t),
\\&\quad
\m y_0(r,t),\m y_1(r,t),\m y_2(r,t),\m y_3(r,t))
\in
(\R^{(n-1)\times(n-1)})^8
.
\end{split}
\end{equation}
We use operator norms for the individual matrices, and the norm of $Z(r,t)$ is the sum of norms of the eight matrices.

Suppose~$Z'(r,t)$ is another solution with the same initial conditions, with all the matrices decorated with a prime.
We know that $Z(0,t)=Z'(0,t)$ for all $t\in I_1$.
Let
\begin{equation}
\rho_0
=
\inf\{r\leq0;Z(s,t)=Z'(s,t)\text{ for all }t\in I_1\text{ and }s\in(r,0)\}.
\end{equation}
To show that $Z=Z'$, assume for a contradiction that $\rho_0>-T$.
By continuity, $Z(\rho_0,t)=Z'(\rho_0,t)$ for all $t\in I_1$.

Using the fundamental theorem of calculus, we find that for any $\rho\in(-T,\rho_0)$ and $t\in I_0$, we have
\begin{equation}
\label{eq:ftoc-z}
Z(\rho,t)-Z'(\rho,t)
=
\int_{\rho_0}^{\rho}(\partial_rZ(r,t)-\partial_rZ'(r,t))\dd r.
\end{equation}
Using equation~\eqref{eq:eom2}, we find formulas for $\partial_rZ(r,t)-\partial_rZ'(r,t)$:
\begin{equation}
\label{eq:r-ders}
\begin{cases}
\partial_r\m j_k(r,t)-\partial_r\m j_k'(r,t)
=
\m y_k(r,t)-\m y_k'(r,t)\\
\partial_r\m y_k(r,t)-\partial_r\m y_k'(r,t)
=
\frac12Q(Z(r,r))\m j_k(r,t)
\\
\qquad\qquad\qquad\qquad\qquad
-\frac12Q(Z'(r,r))\m j_k'(r,t),
\end{cases}
\end{equation}
where we have abbreviated
\begin{equation}
Q(Z^{(\prime)}(r,r))
=
Q(\m j^{(\prime)}(r,r),\m y^{(\prime)}(r,r)).
\end{equation}
We will estimate these derivatives in terms of a suitable norm of $Z-Z'$.

The function~$Q$ defined in~\eqref{eq:Q-def} is smooth and well defined in a neighborhood of any point where~$B_0$ (which plays the role of~$\m y_0$) is invertible.
Therefore there are $\eta,L>0$ so that~$Q$ is $L$-Lipschitz in the closed ball $\bar B(Z(\rho_0,\rho_0),\eta)$.

For any $\delta>0$, let us denote $F_\delta=[\rho_0-\delta,\rho_0]^2$.
Due to continuity, for a sufficiently small $\delta_0>0$ we have
\begin{equation}
\max_{(r,t)\in F_{\delta_0}}\abs{Z(r,t)-Z(\rho_0,\rho_0)}
\leq
\eta
\end{equation}
and
\begin{equation}
\max_{(r,t)\in F_{\delta_0}}\abs{Z'(r,t)-Z(\rho_0,\rho_0)}
\leq
\eta.
\end{equation}
This ensures that the function~$Q$ is $L$-Lipschitz in its parameters when $(r,t)\in F_{\delta_0}$.
It follows that
\begin{equation}
\label{eq:aabs(Q(Z))}
\abs{Q(Z(r,t))}
\leq
\abs{Q(Z(\rho_0,\rho_0))}
+
L\eta
\eqqcolon
\alpha
\end{equation}
for all $(r,t)\in F_{\delta_0}$ and similarly for $\abs{Q(Z'(r,t))}$.

We will keep $\delta_0$ fixed and consider a new parameter $\delta\in(0,\delta_0]$ which we will fix later.
By~\eqref{eq:ftoc-z} and~\eqref{eq:r-ders}, we have
\begin{equation}
\m j_k(\rho,\tau)-\m j_k'(\rho,\tau)
=
\int_{\rho_0}^\rho\left(
\m y_k(r,\tau)-\m y_k'(r,\tau)
\right)\dd r,
\end{equation}
and so
\begin{equation}
\label{eq:j-j'-est}
\max_{F_{\delta}}\abs{\m j_k-\m j_k'}
\leq
\delta\max_{F_{\delta}}\abs{Z-Z'}.
\end{equation}
Similarly, we find
\begin{equation}
\begin{split}
&
\abs{\m y_k(\rho,\tau)-\m y_k'(\rho,\tau)}
\\&=
\frac12\abs{\int_{\rho_0}^\rho\left(
Q(Z(r,r))\m j_k(r,\tau)
-
Q(Z'(r,r))\m j_k'(r,\tau)
\right)\dd r}
\\&\leq
\frac12
\int_{\rho_0}^\rho
\bigg(\abs{
Q(Z(r,r))(\m j_k(r,\tau)-\m j_k'(r,\tau))
}
\\&\quad
+\abs{
(Q(Z(r,r))-Q(Z'(r,r)))\m j_k'(r,\tau)
}\bigg)\dd r
\\&\leq
\frac12
\int_{\rho_0}^\rho
\bigg(
\alpha\abs{\m j_k(r,\tau)-\m j_k'(r,\tau)}
\\&\quad
+L\abs{Z(r,r)-Z'(r,r)}\abs{\m j_k'(r,\tau)}
\bigg)\dd r.
\end{split}
\end{equation}
The ODE~\eqref{eq:eom2} with naive estimates and Gr\"onwall's inequality can be used to establish an a priori estimate $\max_{F_{\delta_0}}\abs{\m j_k'}\leq\beta<\infty$. 
We may then estimate
\begin{equation}
\label{eq:y-y'-est}
\max_{F_{\delta}}\abs{\m y_k-\m y_k'}
\leq
\frac\delta2
(\alpha+L\beta)
\max_{F_{\delta}}\abs{Z-Z'}.
\end{equation}
Combining~\eqref{eq:j-j'-est} and~\eqref{eq:y-y'-est} for all four values of~$k$ gives
\begin{equation}
\label{eq:Z-Z'-est}
\max_{F_{\delta}}\abs{Z-Z'}
\leq
\delta
(4+2(\alpha+L\beta))
\max_{F_{\delta}}\abs{Z-Z'}.
\end{equation}
If we choose
\begin{equation}
\delta
=
\min\left\{\delta_0,\frac1{8+4(\alpha+L\beta)}\right\},
\end{equation}
the estimate~\eqref{eq:Z-Z'-est} gives $\max_{F_{\delta}}\abs{Z-Z'}=0$.
Thus $Z=Z'$ in~$F_\delta$.

This implies that $Q(Z(r,r))=Q(Z'(r,r))$ for all $r\in[\rho_0-\delta,\rho_0]$.
Since $Z(\rho_0,t)=Z'(\rho_0,t)$ for all $t\in I_1$, 
unique solvability of the ODE system~\eqref{eq:eom1} with fixed~$R|_{[\rho_0-\delta,\rho_0]}$ shows that the solutions~$Z(r,t)$ and~$Z'(r,t)$ must coincide for all $(r,t)\in[\rho_0-\delta,\rho_0]\times I_1$.
This contradicts the definition of~$\rho_0$ and shows that indeed $\rho_0=-T$, whether it is a real number or~$-\infty$.

{\bfseries Step 4: Finding~$S$, $K$, and~$R$.}
We have now proven that the Jacobi fields~$\m j$ are uniquely determined in $I_1\times I_1$.
The matrix $\m s(r,t)$ exists when the two points~$\gamma(r)$ and~$\gamma(t)$ are not conjugate, and this happens precisely when~$\m j(r,t)$ is invertible.
The matrix~$\m s(r,t)$ may be computed from~\eqref{eq:DJ=SJ-m}:
\begin{equation}
\label{eq:S=Y/J}
\m s(r,t)
=
(\partial_r\m j(r,t))
\m j(r,t)^{-1}.
\end{equation}
Therefore the matrix~$\m s$ is uniquely determined on $I_1\times I_1$ when it exists from the data.
The same holds immediately for the inverse~$\m k$ as well.

The matrix~$\m k$ is defined near the diagonal $r=t$ and from it the curvature matrix~$\m r(t)$ can be computed by~\eqref{eq:R=Q-m}.
The proof is complete.
\end{proof}

\begin{remark}
The function~$\m j$ corresponding to a geodesic passing through~$U$ is determined by the data uniquely up to a $t$-dependent change of basis.
One can define the matrix~$\m s(r,t)$ by~\eqref{eq:S=Y/J} and it exists precisely when~$\m j(r,t)$ is invertible.
It is sometimes possible to define the matrix~$\m k(r,t)$ in a limit sense (extension by continuity) even when~$\m s(r,t)$ is not invertible.
This is possible if and only if~$\partial_r\m j(r,t)$ is invertible and is given by $\m k=\m j(\partial_r\m j)^{-1}$.
For example, this happens at the diagonal $r=t$.
\end{remark}

\begin{remark}
\label{rmk:JSK-numerical}
The proof suggests an algorithm for solving the Jacobi fields.
One first needs to set up initial values of the Jacobi fields.
This was done in the proof by setting initial values away from conjugate points and propagating the initial data to $r=0$ by the Jacobi equation.
This is possible because~$F|_U$ determines the curvature operator~$R$ on~$I_0$.
Having initial data on the open interval~$I_0$ was only needed to avoid conjugate points in setting up the initial conditions for Jacobi fields.

Then one solves the ODE~\eqref{eq:eom2} iteratively.
If it has been solved for $r\in[\rho_0,0]$ and $t\in I_0$, it is then solved in a small square~$[\rho_0-\delta,\rho_0]^2$ with a corner at $(\rho_0,\rho_0)$ by means of fixed point iteration of the integral formulation
\begin{equation}
\label{eq:eom3}
\begin{cases}
\m j_k(\rho,t)
=
\m j_k(\rho_0,t)+
\int_{\rho_0}^\rho
\m y_k(r,t)
\dd r
\\
\m y_k(\rho,t)
=
\m y_k(\rho_0,t)+
\int_{\rho_0}^\rho
\frac12 Q(\m j(r,r),\m y(r,r))\m j_k(r,t)
\dd r
,
\end{cases}
\end{equation}
where~$k$ ranges from~$0$ to~$3$.
For small enough step size this will converge; the proof essentially shows that the corresponding integral operator is a contraction for small enough step size $\delta>0$.
As a consistency check, one needs to enforce~$\m j(r,r)=0$.
Once a solution has been found, the curvature operator $\frac12 Q(\m j(r,r),\m y(r,r))$ is uniquely determined for $r\in[\rho_0-\delta,\rho_0]$.
This can be used to propagate the Jacobi fields for other values of $t\in I_0$ through the area $r\in[\rho_0-\delta,\rho_0]$.
This extends the solution from $(r,t)\in[\rho_0,0]\times I_0$ to $(r,t)\in[\rho_0-\delta,0]\times I_0$.
\end{remark}


\subsection{Proof of theorem~\ref{thm:RU}}
\label{sec:pf-thm-RU}

The bulk of the proof of our main result was done in lemma~\ref{lma:JSKR-unique}, and it remains to put the pieces together.

\begin{proof}[Proof of theorem~\ref{thm:RU}]
Let $\eps>0$ be such that $[-\eps,\eps]\subset I$ and $\gamma_1({[-\eps,\eps]})\subset U_1$.
Let us write $I_0=(-\eps,\eps)$ and $I_1=I\cap(-\infty,\eps)$.

Up to identifying with the diffeomorphism~$\psi$, the sphere data on the two manifolds agree.
With this identification, $\Sigma_1(r,t)=\Sigma_2(r,t)$ whenever $r\in I_0$ and $t\in I$ with $t<r$.
As the two surfaces $\Sigma_i(r,t)$ coincide, so do their shape operators $S_i(r,t)$.
This holds whenever $t<r$.

When $t>r$ and both times are in~$I_0$, the spheres have ``negative radius''.
This simply means that in the definition of spheres (see definition~\ref{def:vss-M}) the geodesics are followed backwards, corresponding to the exponential map of the reversed Finsler function~$\rev{F}$.
The relevant parts of these spheres are contained in~$U_i$ by the choice of $\eps>0$, and the spheres agree simply because~$\psi$ is isometric.

Therefore we have $S_1(r,t)=S_2(r,t)$ whenever $r\in I_0$, $t\in I_1$, and the shape operators are defined.
This allows us to use lemma~\ref{lma:JSKR-unique}.
It immediately shows that $R_1(t)=R_2(t)$.

The operator $\hat U(t,s)$ is fully determined by~$\hat R(t)$ through the Jacobi equation, so the equality of the solution operators follows from the equality of the curvature operators.
\end{proof}

\begin{remark}
The last step of the uniqueness proof was backwards in light of the proof of lemma~\ref{lma:JSKR-unique}.
The Jacobi fields $J(r,t)$ were the main object whose uniqueness was established and everything else was derived from them.
Indeed, one could extract $U(t,s)$ from $J(t,s)$, but this is a detour in comparison to arguing by the Jacobi equation as we did.
If~$J_1$ and~$J_2$ are the Jacobi fields on~$M_1$ and~$M_2$, respectively, then the lemma gives $J_1(r,t)=J_2(r,t)A(t)$ for some smooth pointwise invertible matrix-valued function $A\colon I_1\to\R^{(n-1)\times(n-1)}$.
\end{remark}

\section{Surface normal coordinates}
\label{sec:snc}
\label{sec:pf-thm-Riemann}

Let $\Sigma\subset M$ be a smooth oriented submanifold of codimension~$1$, and let~$\nu(x)$ be the positively oriented unit normal vector to~$\Sigma$ at $x\in\Sigma$.
We define the surface normal exponential map for~$\Sigma$ as
\begin{equation}
\label{eq:sne-def}
\begin{split}
\phi&\colon\Sigma\times\R\to M,\\
\phi&(x,t)=\gamma_{x,\nu(x)}(t),
\end{split}
\end{equation}
where~$\gamma_{x,v}$ is the unique geodesic with $\gamma_{x,v}(0)=x$ and $\dot\gamma_{x,v}(0)=v$.
This map is illustrated in figure~\ref{fig:sne}.
The surface normal exponential map of~\eqref{eq:sne-def} might not be defined for all $t\in\R$ if the manifold is not geodesically complete.

Note that typically $\gamma_{x,\nu(x)}(t)\neq\exp_x(t\nu(x))$ when $t<0$ if the Finsler geometry is not reversible.
This map~$\phi$ will be relevant for $t<0$, as we follow the arriving geodesics backwards from the spheres in the data towards the sources.

The point $\phi(x,t)\in M$ is said to be a focal point to~$\Sigma$ if the differential $\der\phi(x,t)$ is not bijective.
Therefore every non-focal point has a neighborhood where~$\phi$ is a diffeomorphism and thus gives coordinates on~$M$ in terms of coordinates on~$\Sigma$.
We call these the surface normal coordinates associated with~$\Sigma$.
These coordinates are particularly natural for our problem, as the manifold~$M$ is mostly unknown and the surfaces are given as the data.

Transforming coordinates on the sphere~$\Sigma$ in the data into coordinates on~$M$ was described in more detail in~\eqref{eq:snc-def}.

\begin{lemma}
\label{lma:d-sne}
Let $\Sigma\subset M$ be a smooth oriented submanifold of codimension~$1$ on a geodesically complete Finsler manifold~$M$.
Define the surface normal exponential map $\phi\colon\Sigma\times\R\to M$ as in~\eqref{eq:sne-def}.
This map is smooth.

Take any $x\in\Sigma$ and $t\in\R$, and any $v\in T_x\Sigma$ and $\tau\in\R=T_t\R$.
Then
\begin{equation}
[\der\phi(x,t)](v,\tau)
=
J(t),
\end{equation}
where~$J$ is the Jacobi field along the geodesic~$\gamma_{x,\nu(x)}$ with the initial conditions
\begin{equation}
\begin{split}
J(0)
&=
v
+\tau\nu(x)
\quad\text{and}\\
D_tJ(0)
&=
Sv,
\end{split}
\end{equation}
where~$S$ is the shape operator of~$\Sigma$.
\end{lemma}

\begin{proof}
Smoothness follows from the smoothness of the submanifold and the geodesic flow.

As the direction of the relevant geodesic starting at~$x$ normal to~$\Sigma$ is~$\nu(x)$, the shape operator is a map $S\colon T_x\Sigma\to T_x\Sigma$.
Therefore the initial condition~$D_tJ(0)$ given above has no component parallel to the geodesic $\gamma\coloneqq\gamma_{x,\nu(x)}$.

The derivative in the variable~$t$ is trivial, as it corresponds to moving the endpoint along the geodesic~$\gamma$.
The claim holds for this directional derivative due to
\begin{equation}
[\der\phi(x,t)](0,\tau)
=
\tau\dot\gamma(t)
=
J(t),
\end{equation}
where~$J$ is the Jacobi field (which is simply parallel transported) with initial conditions $J(0)=\tau\dot\gamma(0)$ and $D_tJ(0)=0$.
We thus set $\tau=0$ and focus on differentiation in the variable~$x$.

Let $\sigma\colon(-\eps,\eps)\to\Sigma$ be a smooth curve so that $\sigma(0)=x$ and $\sigma'(0)=v$.
Now $\gamma_s=\gamma_{\sigma(s),\nu(\sigma(s))}$ is a family of geodesics parametrized smoothly by $s\in(-\eps,\eps)$.
The variation of this family at $\gamma=\gamma_0$ is the Jacobi field $J(t)=\partial_s\gamma_s(t)|_{s=0}$.

We then turn to finding the initial conditions of this Jacobi field.
We have $J(0)=\sigma'(0)=v$.
As all the geodesics in the family start normal to~$\Sigma$, the covariant derivative is given by $D_tJ(0)=SJ(0)$ --- a more complicated case was covered in lemma~\ref{lma:1order-Jacobi}.
Therefore~$J$ has the initial conditions as claimed (with $\tau=0$), and so
\begin{equation}
[\der\phi(x,t)](v,0)
=
\partial_s\phi(\sigma(s),t)|_{s=0}
=
J(t)
\end{equation}
as claimed.
\end{proof}

We are now ready to prove our second theorem.

\begin{proof}[Proof of theorem~\ref{thm:Riemann}]
Consider a visible smooth sphere $\Sigma\subset U$, a point $x\in\Sigma$ and let~$S$ be the shape operator of~$\Sigma$ at~$x$.
By lemma~\ref{lma:d-sne} the surface normal exponential map~$\phi$ has a bijective differential at $(x,t)$ if and only if the map $\beta\colon v\to J(t)$ is bijective, where $v\in T_x\Sigma$ and~$J$ is the Jacobi field along~$\gamma_{x,\nu(x)}$ with the initial conditions $J(0)=v$ and $D_tJ(0)=Sv$.
(Both conditions are equivalent with~$t$ not being a focal distance to~$\Sigma$ at~$x$ by definition.)
The local coordinate maps $\alpha_i\colon\Omega\to\Sigma_i$ do not change these properties.

As the shape operator and the solution operator to the Jacobi equation (with respect to a parallel frame) are determined by the sphere data due to theorem~\ref{thm:RU}, the two maps~$\phi_{\alpha_i}$ for $i=1,2$ have bijective differentials at the same points.
This was the first claim.

We will now drop the index~$i$ and show that the sphere data determines the Riemannian metric tensor $g=g_G$ along the vector field~$G$ in the surface normal coordinates given by~$\phi_\alpha$.
To that end, pick any indices $j,k\in\{1,\dots,n\}$.

The last coordinate vector field is $\partial_{x^n}=G$, and so $g_{nn}=\abs{G}=1$.
Let us then find~$g_{nk}$ for $k<n$.
The coordinate vector field $\partial_{x^k}$ coincides with the value of a Jacobi field~$J_k$, for which~$J_k(0)$ is the $k$th basis vector in the local coordinates~$\alpha$ on~$\Sigma$ and $D_tJ_k(0)=SJ_k(0)$.
Here~$S$ is again the shape operator on~$\Sigma$, determined by the data.
This Jacobi field~$J_k(t)$ is normal to the geodesic at all times and so
\begin{equation}
g_{nk}
=
g(\partial_{x^n},\partial_{x^k})
=
g(G,J_k)
=
0.
\end{equation}
Here we leave the points of evaluation implicit to reduce clutter.

If~$j$ and~$k$ are both in $\{1,\dots,n-1\}$, then the coordinate vector fields both correspond to Jacobi fields and
\begin{equation}
g_{jk}
=
g(J_j,J_k).
\end{equation}
As the data determines the two Jacobi fields~$J_j$ and~$J_k$ in an orthonormal frame (theorem~\ref{thm:RU}), we know the map $t\mapsto g_{\dot\gamma(t)}(J_j(t),J_k(t))$.
Evaluation of this map gives the components of the metric tensor in directions normal to the vector field~$G$.

The metric tensor in the direction parallel to~$G$ is determined more straightforwardly as we saw above.
This completes the proof.
\end{proof}






\subsection*{Acknowledgements}

MVdH gratefully acknowledges the support from the Simons Foundation under the MATH + X program, and the National Science Foundation under grant DMS-1559587.
JI was supported by the Academy of Finland (decisions 295853, 332890, and 336254). 
Much of the work was completed during JI's visits to Rice University, and he is grateful for hospitality and support from the Simons Foundation.
ML was supported by Academy of Finland (decisions
320113, 
303754, 
318990, and 
312119
).
The authors thank Teemu Saksala for numerous remarks and discussions, Negar Erfanian for help with drawing figures~\ref{fig:geodesic-and-normals} and~\ref{fig:sne}, and David Mis for help with drawing figure~\ref{fig:initial-conditions}.
The authors are grateful for the feedback from the anonymous referees.

\end{document}